\documentclass[reqno,11pt]{amsart}
\usepackage{latexsym}
\usepackage{amsfonts}
\usepackage{amssymb}
\usepackage{mathrsfs}
\usepackage[all]{xy}
\usepackage{bbm}
\usepackage{setspace}

\allowdisplaybreaks

\newtheorem{theorem}{Theorem}[section]

\newtheorem{lemma}[theorem]{Lemma}
\newtheorem{proposition}[theorem]{Proposition}
\theoremstyle{definition}
\newtheorem{definition}[theorem]{Definition}
\newtheorem{remark}[theorem]{Remark}
\numberwithin{equation}{section}

\newcommand{\RR}{\mathbb{R}}
\newcommand{\NN}{\mathbb{N}}
\newcommand{\QQ}{\mathbb{Q}}
\newcommand{\ZZ}{\mathbb{Z}}

\newcommand{\cP}{\mathscr{P}}

\newcommand{\cO}{\mathcal{O}}

\newcommand{\sB}{\mathscr{B}}

\newcommand{\Bad}{\mathbf{Bad}}

\newcommand{\rr}{\mathbf{r}}
\newcommand{\bx}{\mathbf{x}}
\newcommand{\bs}{\mathbf{s}}
\newcommand{\bG}{\mathbf{G}}

\newcommand{\diag}{\mathrm{diag}}

\newcommand{\SL}{\mathrm{SL}}
\newcommand{\Ad}{\mathrm{Ad}}
\newcommand{\Res}{\mathrm{Res}}
\newcommand{\ggm}{G/\Gamma}

\newcommand\hs{homogeneous space}

\title[Bounded orbits on homogeneous spaces]{Bounded orbits of Diagonalizable Flows on finite volume quotients of products of $\SL_2(\RR)$}

\begin{document}
\author{Jinpeng An}
\address{LMAM, School of Mathematical Sciences, Peking University, Beijing, 100871, China}
\email{anjinpeng@gmail.com}
\thanks{An is supported by NSFC grant 11322101}
\author{Anish Ghosh}
\address{School of Mathematics, Tata Institute of Fundamental Research, Homi Bhabha Road, Colaba, Mumbai 400005, India}
\email{ghosh@math.tifr.res.in}
\thanks{Ghosh is supported by a UGC grant, a CEFIPRA grant and a DST Swarnajayanti fellowship}
\author{Lifan Guan}
\address{Department of Mathematics, University of York, Heslington, York, YO10 5DD, England \\
Current address: Georg-August Universit\"at G\"ottingen\\
Mathematisches Institut\\
Bunsenstrasse 3-5\\
D-37073 G\"ottingen, Germany}
\email{guanlifan@gmail.com}
\thanks{Guan was supported by EPSRC grant EP/J018260/1}
\author{Tue Ly}
\address{Department of Mathematics, Brandeis University, Waltham, MA 02453, USA}
\email{lntue.h@gmail.com}

\begin{abstract}
We prove a number field analogue of a result of C. McMullen that the set of badly approximable numbers is absolute winning. We also prove a weighted version. We use this to prove an instance of a conjecture of An, Guan and Kleinbock \cite{AGK}. Namely, let $G := \SL_2(\RR) \times \dots \times \SL_2(\RR) $ and $\Gamma$ be a lattice in $G$. We show that the set of points on $G/\Gamma$ whose orbits under a one parameter Ad-semisimple subgroup of $G$ are bounded, form a \emph{hyperplane absolute winning} set.
\end{abstract}

\maketitle

\section{Introduction}
Let $G$ be a Lie group. We will say that $g \in G$ is $\Ad$-semisimple if $\Ad_g$ is diagonalizable over $\mathbb{C}$ and $\Ad$-diagonalizable if $\Ad_g$ is diagonalizable over $\RR$. We say that a one parameter subgroup $F$ is $\Ad$-semisimple (resp. $\Ad$-diagonalizable) if all the elements from $F$ are $\Ad$-semisimple (resp. $\Ad$-diagonalizable). In this paper, we prove the following theorem that verifies some cases of \cite[Conjecture 7.1]{AGK}:
\begin{theorem}\label{main thm}
Let $G = \SL_2(\RR) \times \dots \times \SL_{2}(\RR)$ be a finite product of copies of $\SL_{2}(\RR)$ and let $\Gamma$ be a lattice subgroup of $G$. Then for any one parameter $\Ad$-semisimple subgroup $F=\{g_t:t\in \RR\}$ of $G$, the set
\begin{equation*}
E(F):=\{x\in \ggm: Fx \text{ is bounded}\}
\end{equation*}
is Hyperplane Absolute Winning (HAW).
\end{theorem}

When  the subgroup $F$ is unbounded and the lattice $\Gamma$ is irreducible, the action of $F$  on the finite volume homogeneous space $G/\Gamma$ is ergodic and as a consequence, the set $E(F)$ has zero (Haar) measure. One consequence of the HAW property proved in Theorem \ref{main thm} is that it nevertheless is \emph{thick}, i.e. has full Hausdorff dimension at any point of the space. In fact, the HAW property is much richer and HAW sets exhibit many more interesting properties in addition to being thick.  The conjecture of An, Guan and Kleinbock predicts that $E(F)$ is HAW for $G$ any Lie group, $\Gamma$ any lattice in $G$ and $F$ any Ad-diagonalizable subgroup of $G$. In the same paper, this conjecture is verified for $G = \SL_3(\RR)$ and $\Gamma = \SL_3(\ZZ)$. This type of result goes back to the work of S. G. Dani \cite{Da2}, from whose work the winning property can be verified for real rank $1$ Lie groups. The AGK conjecture upgrades the winning property to HAW. As observed by Dani, the study of bounded orbits of diagonalizable flows on homogeneous spaces is intimately related to the study of \emph{badly approximable} numbers or matrices. This connection will also be important in the present work. In particular, along the way to proving the main theorem, we will prove (cf. Proposition \ref{main lemma} below) a number field analogue of C. McMullen's \cite{McM} result that badly approximable numbers form an absolute winning set, itself a strengthening of W. M. Schmidt's \cite{Sc3} result that the set of badly approximable numbers has full Hausdorff dimension. We note however, that in contrast to one dimensional badly approximable numbers, the number field case is richer and allows for the study of weighted approximation. In this sense, the number field case possesses some features of the high dimensional problem.  We believe this result to be of independent interest.

Following Dani's influential paper, there have been significant advances both in the understanding of bounded orbits of diagonalizable flows on homogeneous spaces, as well as in the study of badly approximable numbers and vectors. On the homogeneous side, we mention Margulis' conjecture, resolved by Kleinbock and Margulis \cite{KM}, and the work of  Kleinbock \cite{Kl}, Kleinbock-Weiss \cite{KW1, KW3} and An-Guan-Kleinbock \cite{AGK}. On the number theoretic side, we mention W. M. Schmidt's conjecture, resolved by Badziahin, Pollington and Velani \cite{BPV} and their subsequent strengthening in different contexts, by An \cite{An1, An2}, Beresnevich \cite{Be} and An, Beresnevich and Velani \cite{ABV}. We refer the reader to these works for the history of the problems as well as a more comprehensive list of results and references. Pertinent to the present work are the papers \cite{EGL} of Einsiedler, Ghosh and Lytle and \cite{KL} of Kleinbock and Ly where some special cases of Theorem \ref{main thm} were established, namely the cases
\begin{enumerate}
\item $G = \SL_{2}(\RR) \times \dots \times \SL_{2}(\RR), \Gamma = \SL_{2}(\cO_K)$ and $F = \{g_t~:~t \in \RR\}$ where
$$g_t := \left(\begin{pmatrix} e^{t} & 0\\0 & e^{-t}\end{pmatrix}, \dots, \begin{pmatrix} e^{t} & 0\\0 & e^{-t}\end{pmatrix} \right).$$
In \cite{EGL}, $E(F)$ was shown to be winning for Schmidt's game. In fact, a more general result, involving points in $C^1$ curves whose forward orbits are bounded, was proved. Subsequently in \cite{KL}, D. Kleinbock and the fourth named author proved a property stronger than HAW. In particular, this case of Theorem \ref{main thm} is due to Kleinbock and Ly.\\
\item In \cite{EGL}, the case of $K$ a real quadratic field, $G = \SL_{2}(\RR) \times \SL_{2}(\RR)$, $\Gamma = \SL_{2}(\cO_K)$ and $F = \{g_t~:~t \in \RR\}$ where
$$g_t := \left(\begin{pmatrix} e^{r_{\sigma_1}t} & 0\\0 & e^{-r_{\sigma_1}t}\end{pmatrix}, \begin{pmatrix} e^{r_{\sigma_2}t} & 0\\0 & e^{-r_{\sigma_2}t}\end{pmatrix} \right)$$
was also considered. Here $r_{\sigma_i} \geq 0$ and $r_{\sigma_1} +  r_{\sigma_2} = 1$. The corresponding $E(F)$ was shown to be winning for Schmidt's game.
\end{enumerate}
In \S \ref{S:2} we record preliminaries on the hyperplane absolute game and the hyperplane potential game. These are variants of the classical game introduced by W. M. Schmidt \cite{Sc1}. The subsequent two sections are devoted to the proof of a special case of Theorem \ref{main thm}, namely when $G$ equals product of $d$ copies of $\SL_2(\RR)$ and $\Gamma$ equals $\SL_{2}(\cO_K)\subset G$ where $K$ is a totally real field of degree $d$ over $\QQ$, $\cO_K$ is its ring of integers and $F$ is an arbitrary Ad-diagonalizable subgroup
of $G$. This particular case of our theorem is connected to Diophantine approximation of vectors in $\RR^d$ by rationals in the number field $K$. Indeed, this case is the generalisation of the result of \cite{EGL} in (2) above. This case forms the bulk of our paper and is intimately connected to the number field analogue of McMullen's result. We use a transference (``the Dani correspondence") to relate this case to the HAW property of certain vectors badly approximable by rationals in $K$ and prove this latter property. Finally we use the structure theory of Lie groups and Margulis arithmeticity theorem to conclude the proof of Theorem \ref{main thm}.
We conclude the introduction with some remarks:
\begin{enumerate}
\item It is plausible that the method of proof developed in the present paper can be used to deal with the case where $G$ consists of products of $\SL_2(\RR)$ and $\SL_2(\mathbb{C})$. Indeed the main argument would then be carried out with an arbitrary number field rather than a totally real number field.
\item Proposition \ref{main lemma} below, can be formulated for arbitrary number fields rather than just totally real ones. The proof is identical to the one presented here; we have restricted ourselves to totally real fields for notational ease.
\item In Theorem 4.2 in his thesis \cite{Ly}, the last named author proved a more general version of Proposition 3.5 below. Specifically, the notion of winning used is slightly more general and a higher dimensional analogue of $\Bad(K,\rr)$ (defined below) is considered. This result can be used to verify  \cite[Conjecture 7.1]{AGK} in some more cases, namely for certain Ad-semisimple one parameter flows on some special quotients of products of $\SL_n(\RR)$. Simultaneously and independently, the first three named authors established Theorem \ref{main thm}. We then agreed to join hands in this paper.
\end{enumerate}

\subsection*{Acknowledgements} This work was initiated during a visit by Ghosh to Peking University. He is very grateful to the host for the invitation and the hospitality. Subsequent progress was made during a visit by the first three authors to Oberwolfach. We would like to thank the MFO for the excellent working conditions and V. Beresnevich and S. Velani for the invitation. We thank the referee for several helpful suggestions which have improved the paper.

\section{Preliminaries on Schmidt games}\label{S:2}

In this section, we will recall definitions of certain recent variants of Schmidt games, namely, the hyperplane absolute game and the hyperplane potential game. We follow the exposition in \cite{AGK}. They are both variants of the $(\alpha,\beta)$-game introduced by Schmidt in \cite{Sc1}. Since we do not make a direct use of the $(\alpha,\beta)$-game in this paper, we omit its definition here and refer the interested reader to \cite{Sc1,S2}. Instead, we list here some nice properties of the $\alpha$-winning sets:
\begin{enumerate}
\item If the game is played on a Riemannian manifold, then any $\alpha$-winning set is thick.
\item The intersection of countably many $\alpha$-winning sets is $\alpha$-winning.
\end{enumerate}

\subsection{Hyperplane absolute game}
The hyperplane absolute game was introduced in \cite{BFKRW}. It is played on a Euclidean space $\RR^{d}$.
Given a hyperplane $L$ and a $\delta>0$, we denote by $L^{(\delta)}$ the $\delta$-neighborhood of $L$, i.e.,
$$L^{(\delta)}:=\{\bx\in\RR^{d}:\mathrm{dist}(\bx,L)<\delta\}.$$ For $\beta\in(0,\frac{1}{3})$, the
\emph{$\beta$-hyperplane absolute game} is defined as follows. Bob starts by choosing a
closed ball $B_0\subset \RR^d$ of radius $\rho_0$. In the $i$-th turn, Bob chooses a closed ball $B_i$ with radius $\rho_i$, and then
Alice chooses a  hyperplane neighborhood $L_i^{(\delta_i)}$ with $\delta_i\le\beta \rho_i$. Then in the $(i+1)$-th turn, Bob chooses a closed ball $B_{i+1}\subset B_i\setminus L_i^{(\delta_i)}$ of radius $\rho_{i+1}
\geq\beta \rho_{i}$. By this process there is a nested sequence of closed balls
$$B_0\supseteq B_1\supseteq B_2\supseteq \cdots.$$
We say that a subset $S\subset \RR^{d}$ is \emph{$\beta$-hyperplane absolute
winning} (\emph{$\beta$-HAW} for short) if no matter how Bob plays, Alice can ensure that
$$\bigcap_{i=0}^\infty B_i\cap S\neq \emptyset.$$ We say $S$ is \emph{hyperplane absolute winning}
(\emph{HAW} for short) if it is $\beta$-HAW for any $\beta\in(0,\frac{1}{3})$.

We have the following lemma collecting the basic properties of $\beta$-HAW subsets and HAW subsets of $\RR^d$ (\cite{BFKRW}, \cite{KW3}):
\begin{lemma}\label{haw-property}
\begin{enumerate}
\item A HAW subset is always $\frac{1}{2}$-winning.
\item Given $\beta,\beta'\in(0,\frac{1}{3})$, if $\beta\geq\beta'$, then any $\beta'$-HAW set is $\beta$-HAW.
\item A countable intersection of HAW sets is again HAW.
\item Let $\varphi:\RR^{d}\to \RR^{d}$ be a $C^1$ diffeomorphism. If $S$ is a HAW set, then so is
$\varphi(S)$.
\end{enumerate}
\end{lemma}

The notion of HAW was extended to subsets of $C^1$ manifolds in \cite{KW3}. This is done in two steps. First, one defines the hyperplane absolute game on an open subset $W\subset \RR^d$. It is defined just as the hyperplane absolute game on $\RR^d$, except for requiring that Bob's first move $B_0$ be contained in $W$. Now, let $M$ be a $d$-dimensional $C^1$ manifold, and let $\{(U_{\alpha},\phi_{\alpha})\}$ be a $C^1$ atlas on $M$. A subset $S\subset M$ is said to be HAW on $M$ if for each $\alpha$, $\phi_{\alpha}(S\cap U_{\alpha})$ is HAW on $\phi_{\alpha}(U_{\alpha})$. The definition is independent of the choice of atlas by the property (4) listed above. We have the following lemma that collects the basic properties of HAW subsets of a $C^1$ manifold (cf. \cite{KW3}).

\begin{lemma}\label{haw-manifold}

\begin{enumerate}
\item HAW subsets of a $C^1$ manifold are thick.
\item A countable intersection of HAW subsets of a $C^1$ manifold is again HAW.
\item Let $\phi:M\rightarrow N$ be a diffeomorphism between $C^1$ manifolds, and let $S\subset M$ be a HAW subset of $M$. Then $\phi(S)$ is a HAW subset of $N$.
\item Let $M$ be a $C^1$ manifold with an open cover $\{U_{\alpha}\}$. Then, a subset $S\subset M$ is HAW on $M$ if and only if $S\cap U_{\alpha}$ is HAW on $U_{\alpha}$ for each $\alpha$.
\item Let $M_1,M_2$ be $C^1$ manifolds, and let $S_i\subset M_i \ (i=1,2)$ be  HAW subsets of $M_i$. Then $S_1\times S_2$ is a HAW subset of $M_1\times M_2$.
\end{enumerate}
\end{lemma}
\begin{proof}
  Indeed, everything except (5) is proved in \cite{KW3}. So we provide a proof of (5) here. According to \cite[Lemma 2.2(v)]{AGK}, both the set $S_1\times M_2$ and $S_2 \times M_1$ are HAW. Thus the set $S_1\times S_2$ is HAW by (2).
\end{proof}

\subsection{Hyperplane potential game}

The hyperplane potential game was introduced in \cite{FSU} and also defines a class of subsets of $\RR^{d}$
called \emph{hyperplane potential winning} (\emph{HPW} for short) sets. The following lemma allows one to prove
the HAW property of a set $S\subset \RR^d$ by showing that it is winning for the hyperplane potential game. And this is exactly the game we will use in this paper.
\begin{lemma}\emph{(cf. \cite[Theorem C.8]{FSU})}\label{HPW}
A subset $S$ of $\RR^{d}$ is HPW if and only if it is HAW.
\end{lemma}

The hyperplane potential game involves two parameters $\beta\in(0,1)$ and $\gamma>0$.  Bob starts the game
 by choosing a closed ball $B_0\subset \RR^{d}$ of radius $\rho_0$. In the $i$-th turn,
Bob chooses a closed ball $B_i$ of radius $\rho_i$, and then Alice chooses a countable family of hyperplane
neighborhoods $\{L_{i,k}^{(\delta_{i,k})}: k\in \NN\}$ such that
\begin{equation*}
\sum_{k=1}^\infty \delta_{i,k}^\gamma\le(\beta \rho_{i})^\gamma.
\end{equation*}
Then in the $(i+1)$-th turn, Bob chooses a closed ball $B_{i+1}\subset B_i$ of radius $\rho_{i+1}\ge\beta \rho_{i}$. By
this process there is a nested sequence of closed balls
$$B_0\supseteq B_1\supseteq B_2\supseteq \cdots.$$
We say a subset $S\subset \RR^{d}$ is \emph{$(\beta,\gamma)$-hyperplane potential winning}
(\emph{$(\beta,\gamma)$-HPW} for short) if no matter how Bob plays, Alice can ensure that
$$\bigcap_{i=0}^\infty B_i\cap\Big(S\cup\bigcup_{i=0}^\infty\bigcup_{k=1}^\infty
L_{i,k}^{(\delta_{i,k})}\Big)\ne\emptyset.$$ We say $S$ is \emph{hyperplane potential winning}
(\emph{HPW} for short) if it is $(\beta,\gamma)$-HPW for any $\beta\in(0,1)$ and $\gamma>0$.

\section{a special case}
This and the next section are devoted to prove a special case of Theorem \ref{main thm}. We begin by introducing some notation. Let $K$ be a totally real field of degree $d$ over $\QQ$, $\cO_K$ its ring of integers, and $S$ be the set of field embeddings $K\hookrightarrow\RR$. Then we have $|S|=d.$ Set
\begin{equation*}
\theta: K\to \prod_{\sigma\in S} \RR, \quad \theta(p)=(\sigma(p))_{\sigma\in S}.
\end{equation*}
Let $\Res_{K/\QQ}$ denote Weil's restriction of scalar's functor. It is well known (\cite[Theorem 7.8]{BH}) that the group $\Res_{K/\QQ}\SL_2(\ZZ)$
is a lattice in  $\Res_{K/\QQ}\SL_2(\RR)$. The latter coincides with the product of $d$ copies of $\SL_2(\RR)$. For simplicity, in this section and the next, we set
$$G=\Res_{K/\QQ}\SL_2(\RR)=\prod_{\sigma\in S}\SL_2(\RR), \quad \Gamma=\Res_{K/\QQ}\SL_2(\ZZ).$$
It follows from the definition that the subgroup $\Res_{K/\QQ}\SL_2(\ZZ)$ coincides with the subgroup $\theta(\SL_2(\cO_K))$, where $\theta$ is the map defined by $\theta(g)=(\sigma(g))_{\sigma\in S}$.
Now we are ready to state the following special case of Theorem \ref{main thm}.

\begin{proposition}\label{main prop}
Let $\rr\in \RR^d$ be a real vector with $r_{\sigma}\ge 0$ for $\sigma\in S$ and $\sum_{\sigma\in S}r_{\sigma}=1$, set
\begin{equation}
g_{\rr}(t):=  \left(\begin{pmatrix}e^{r_{\sigma}t}&0\\0&e^{-r_{\sigma}t}\end{pmatrix}\right)_{\sigma\in S}
\end{equation}
and $F_{\rr}=\{g_{\rr}(t):t\in \RR\}$, then the set
\begin{equation*}
E(F_{\rr}):=\{x\in \ggm: F_{\rr}x \text{ is bounded }\}
\end{equation*}
is HAW.
\end{proposition}

Proposition \ref{main prop} will be proved by studying the set
\begin{equation*}
E(F_{\rr}^+):=\{x\in \ggm: F_{\rr}^+x \text{ is bounded }\},
\end{equation*}
where $F_{\rr}^+ =\{g_{\rr}(t):t\ge 0\}$.


We will fix $\rr$ in this and the next section. Set
$$S_1=\{\sigma\in S:r_{\sigma}>0\}, \text{ and } S_2=S\setminus S_1.$$
Assume $|S_1|=d_1, |S_2|=d_2$.  Choose and fix $\omega\in S$ with $r_{\omega}=r$, where
$$r=\max_{\sigma\in S}r_{\sigma}. $$
 Define a weighted norm, called the \emph{$\rr$-norm}, on $\prod_{\sigma\in S}\RR$ by
$$\|\bx\|_{\rr}=\max_{\sigma\in S_1}|x_{\sigma}|^{\frac{1}{r_{\sigma}}}.$$
\begin{definition}
Say a vector $\bx=(x_{\sigma})_{\sigma\in S} \in \prod_{\sigma\in S}\RR$ is \emph{$(K,\rr)$-badly approximable } if
\begin{equation*}
\inf_{\substack{q\in \cO_K\setminus \{0\}\\p\in \cO_K }} \max \left\{\max_{\sigma\in S_1} \|q\|_{\rr}^{r_{\sigma}}|\sigma(q)x_{\sigma}+\sigma(p)|,
\max_{\sigma\in S_2} \max\{|\sigma(q)x_{\sigma}+\sigma(p)|,|\sigma(q)|\} \right\}>0.
\end{equation*}
The set of $(K,\rr)$-badly approximable vectors is denoted as $\Bad(K,\rr)$.
\end{definition}
\begin{remark}
The notation of $(K,\rr)$-badly approximable vector is the weighted case of $K$-badly approximable vector introduced in \cite{EGL}.
\end{remark}
Denote the strictly upper triangular subgroup of $G$ as $H$, and note that $H$ can be identified with  $\prod_{\sigma\in S}\RR$ through the map:
 $$u:\prod_{\sigma\in S}\RR \rightarrow \prod_{\sigma\in S}\SL_2(\RR), \quad u((x_{\sigma})_{\sigma\in S})= \left(\begin{pmatrix}1&x_{\sigma}\\0&1\end{pmatrix}\right)_{\sigma\in S}$$

Then we have the following correspondence between $(K,\rr)$-badly approximable vectors and bounded $F_{\rr}^+$ trajectories known in the literature as the Dani correspondence.
\begin{proposition}\label{dani}
 A vector $\bx=(x_{\sigma})_{\sigma\in S}$ is $(K,\rr)$-badly approximable if and only if the trajectory $F_{\rr}^+u(\bx)\Gamma$ is bounded in $\ggm$.
In other words,
\begin{equation}\label{e:dani}
\Bad(K,\rr)= u^{-1}\big(\pi^{-1}(E(F_{\rr}^+))\cap H\big),
\end{equation}
where $\pi$ denotes the projection $G\rightarrow \ggm$.
\end{proposition}

\begin{proof}
For simplicity, denote the elements in $S$ as $\{\sigma_1,\dots,\sigma_d\}$ and the weights $r_{\sigma_i}$ as $r_i$. Without loss of generality, we may assume that $r_i>0$ for $1\le i\le d_1$ and $r_i=0$ for $d_1<i\le d$.
It is easily seen that $D_K^{-\frac{1}{2d}}\theta(\cO_K)$ is a unimodular lattice of $\RR^d$, where $D_K$ is the discriminant of $K$. Write the lattice $D_K^{-\frac{1}{2d}}\theta(\cO_K)\times D_K^{-\frac{1}{2d}}\theta(\cO_K)\subset \RR^{2d}$ simply as $L_K$. Then define a homomorphism $\psi: G \to \SL_{2d}(\RR)$ by
\begin{equation}\label{def-psi}
 \psi(g)_{ij}=\begin{cases} a_i, & \text{ if } 1\le i=j\le d, \\
b_i, & \text{ if } 1\le i=j-d\le d, \\  c_{i-d}, & \text{ if } 1\le i-d=j\le d, \\  d_{i-d}, & \text{ if } d+1\le i=j\le 2d, \\ 0, & \text{ otherwise, } \end{cases} \qquad
 \text{ where } g=\left(\begin{pmatrix}a_i&b_i\\ c_i&d_i\end{pmatrix}\right)_{1\le i\le d}.
\end{equation}
Now we claim that
\begin{equation}\label{claim-fix-gk}
\{g\in G: \psi(g)L_K=L_K\}=\Gamma
\end{equation}
Let $g$ be as in \eqref{def-psi}. First, we focus on the study of $(g)_i:=\begin{pmatrix}a_i&b_i\\ c_i&d_i\end{pmatrix}$. If $\psi(g)L_K=L_K$, it follows that $a_1\sigma_1(k)+b_1\sigma_1(k'), c_1\sigma_1(k)+d_1\sigma_1(k')\in \sigma_1(\cO_K)$ for all $k,k'\in \cO_K$. By choosing $k$ or $k'$ to be $0$, we can show that $f\sigma_1(\cO_K)=\sigma_1(\cO_K)$ for $f=a_1,b_1,c_1,d_1$. Hence it follows from the definition of $\cO_K$ that the matrix $(g)_1$ has all its entries in $\sigma_1(\cO_K)$. Consequently, $$(g)_1\in M_{2\times 2}(\sigma_1(\cO_K))\cap \SL_2(\RR)=\sigma_1(\SL_2(\cO_K)).$$
Then since an element in $\SL_2(\RR)$ is uniquely determined by its action on $\RR^2$, it follows that, if $\psi(g)L_K=L_K$, then  $(g)_i=\sigma_i\sigma_1^{-1}\big((g)_1\big)$. This shows that $\psi(g)\in \theta(\SL_2(\cO_K))=\Gamma$. On the other hand, it is clear that, if $g\in \Gamma$, $\psi(g)$ fixes $L_K$. This completes the proof of claim \eqref{claim-fix-gk}.

As $\Gamma$ is a lattice in $G$, in view of claim \eqref{claim-fix-gk} and \cite[Theorem 1.13]{Rag}, we find that the embedding
$$\phi: \ggm\to  \SL_{2d}(\RR)/\SL_{2d}(\ZZ),\quad \phi(g\Gamma)=\psi(g)L_K$$
is a proper map. Note that here we use the fact that the space $\SL_{2d}(\RR)/\SL_{2d}(\ZZ)$ is the space of unimodular lattices in $\RR^{2d}$ implicity. Hence it follows that:
\begin{equation}\label{e:embed}
\text{ $F_{\rr}^+u(\bx)\Gamma$ is bounded in $\ggm$ } \Longleftrightarrow \text{ $\psi(F_{\rr}^+u(\bx))L_K$ is bounded in $\SL_{2d}(\RR)/\SL_{2d}(\ZZ)$}.
\end{equation}

Note that we have
 $$\psi(g_{\rr}(t))=\diag(e^{r_1t},\dots,e^{r_dt},e^{-r_1t},\dots,e^{-r_dt}),$$
 and
 $$\psi(u(\bx))=\begin{pmatrix}I_d& \diag(\bx)\\ & I_d \end{pmatrix}, \text{ where } \diag(\bx)=\diag(x_1,\dots,x_d).$$
 In view of Mahler's criterion and \eqref{e:embed}, we have
 \begin{align*}
  &\ \text{ $F_{\rr}^+u(\bx)\Gamma$ is bounded in $\ggm$ } \\
  \Longleftrightarrow &\ \text{ $\psi(F_{\rr}^+u(\bx))L_K$ is bounded in $\SL_{2d}(\RR)/\SL_{2d}(\ZZ)$} \\
  \Longleftrightarrow &\  \inf_{p,q\in \cO_K}\inf_{t>0} \max \left\{
  \begin{aligned} &\max_{ 1\le i\le d_1} \max\{e^{r_it}|\sigma_{i}(q)x_i+\sigma_{i}(p)|, e^{-r_it}|\sigma_{i}(q)|\}, \\
  &\max_{ d_1< i\le d} \max\{|\sigma_{i}(q)x_i+\sigma_{i}(p)|, |\sigma_{i}(q)|\} \end{aligned} \right\}>0 \\
  \Longleftrightarrow &\ \inf_{\substack{q\in \cO_K\setminus \{0\}\\p\in \cO_K }}\inf_{t>0} \max \left\{
  \begin{aligned} &\max_{ 1\le i\le d_1} \max\{e^{t}|\sigma_{i}(q)x_i+\sigma_{i}(p)|^{\frac{1}{r_i}}, e^{-t}|\sigma_{i}(q)|^{\frac{1}{r_i}}\}, \\
  &\max_{ d_1< i\le d} \max\{|\sigma_{i}(q)x_i+\sigma_{i}(p)|, |\sigma_{i}(q)|\} \end{aligned} \right\}>0 \\
  \Longleftrightarrow &\ \inf_{\substack{q\in \cO_K\setminus \{0\}\\p\in \cO_K }}\inf_{t>0} \max \left\{
  \begin{aligned} &\max \{\max_{ 1\le i\le d_1} e^{t}|\sigma_{i}(q)x_i+\sigma_{i}(p)|^{\frac{1}{r_i}}, e^{-t}\|q\|_{\rr}\}, \\
  &\max_{ d_1< i\le d} \max\{|\sigma_{i}(q)x_i+\sigma_{i}(p)|, |\sigma_{i}(q)|\} \end{aligned} \right\}>0 \\
  \Longleftrightarrow &\ \inf_{\substack{q\in \cO_K\setminus \{0\}\\p\in \cO_K }} \max \left\{
  \begin{aligned} &\|q\|_{\rr}\left(\max_{ 1\le i\le d_1}|\sigma_{i}(q)x_i+\sigma_{i}(p)|^{\frac{1}{r_i}}\right), \\
  &\max_{ d_1< i\le d} \max\{|\sigma_{i}(q)x_i+\sigma_{i}(p)|, |\sigma_{i}(q)|\} \end{aligned} \right\}>0 \\
  \Longleftrightarrow &\ \inf_{\substack{q\in \cO_K\setminus \{0\}\\p\in \cO_K }} \max \left\{
  \begin{aligned} &\max_{ 1\le i\le d_1}\|q\|_{\rr}^{r_i}|\sigma_{i}(q)x_i+\sigma_{i}(p)|, \\
  &\max_{ d_1< i\le d} \max\{|\sigma_{i}(q)x_i+\sigma_{i}(p)|, |\sigma_{i}(q)|\} \end{aligned} \right\}>0.
  \end{align*}
This completes the proof.
\end{proof}

  Now we are ready to state the following Proposition which establishes a weighted number field analogue of McMullen's result. The proof of the Proposition is postponed to the next section.
    \begin{proposition}\label{main lemma}
  $\Bad(K,\rr)$ is HAW.
  \end{proposition}
\begin{proof}[Proof of Proposition \ref{main prop} modulo Proposition \ref{main lemma}]
Write $$P:=\prod_{\sigma\in S} \left(\begin{pmatrix}*&0\\ *&*\end{pmatrix}\right)_{\sigma\in S}.$$
As for any $p\in P$ the set $\{\Ad(g)p: g\in F_{\rr}^+\}$ is bounded,  we have
\begin{equation}\label{e:ph-product}
\Lambda\in E(F_{\rr}^+) \Longleftrightarrow p\Lambda\in E(F_{\rr}^+) \quad \forall p\in P.
\end{equation}
We claim that
\begin{equation}\label{e:surj-proj}
\pi(PH)=\ggm.
\end{equation}
 Indeed, according to the Bruhat decomposition, the set $PH$ is Zariski open in $G$. Suppose to the contrary that $g\Gamma\notin \pi(PH)$ for some $g\in G$, then we will have $\Gamma \cap g^{-1}PH =\emptyset$. This contradicts the Borel density theorem, hence proves our claim.

To prove $E(F_{\rr}^+)$ is HAW, it suffices to prove that for any $\Lambda\in \ggm$, there exists a neighborhood $\Omega$ of $\Lambda$ in $\ggm$ such that the set $E(F_{\rr}^+)\cap \Omega$ is HAW. In view of \eqref{e:surj-proj}, we can find $p_0\in P $ and $u_0\in H$ such that $p_0u_0\Gamma=\Lambda$. Then choose a neighborhood $\Omega_P$ (resp. $\Omega_H$) of $p_0$ (resp. $u_0$) in $P$ (resp. $H$) small enough that the map
$\phi: \Omega_P\times \Omega_H\to \ggm, (p,u) \mapsto pu\Gamma$ is a diffeomorphism onto its image $\Omega$. Hence we are reduced to proving that the set
$$\phi^{-1}(E(F_{\rr}^+)\cap \Omega)=\{(p,u)\in \Omega_P\times \Omega_H: pu\Gamma\in E(F_{\rr}^+)\}$$
is HAW. In view of \eqref{e:ph-product}, the set defined above coincides with
\begin{equation}\label{e:set-ph}
\Omega_P\times \big(\pi^{-1}(E(F_{\rr}^+))\cap \Omega_H\big)
\end{equation}
and the HAW property of the set \eqref{e:set-ph} follows from \eqref{e:dani} and Proposition \ref{main lemma}. Let $F_\rr^-=\{g_t: t\le 0\}$, and let $\rho$ be the automorphism of $G$ which sends $g$ to ${}^tg^{-1}$ on each $\SL_2(\RR)$ component. Then it is clear that $\rho(F_\rr^+)=F_\rr^-$ and $\rho$ fixes $\Gamma$, hence induces a diffeomorphism of $\ggm$, also denoted as $\rho$. As a consequence, $E(F_\rr^-)=\rho(E(F_{\rr}^+)) $ is HAW and so is $E(F)=E(F_\rr^-) \cap E(F_\rr^+)$. This completes the proof.
\end{proof}

\section{Proof of Proposition \ref{main lemma}}
First we introduce another formulation of the set $\Bad(K,\rr)$.
For $\varepsilon>0$, set
$$\cO_K(\rr,\varepsilon)=\{q\in \cO_K\setminus \{0\}: \max_{\sigma\in S_2} |\sigma(q)|\le \varepsilon\}.$$
For $(p,q)\in \cO_K\times \cO_K(\rr,\varepsilon)$, define
 $$\Delta_{\varepsilon}(p,q)= \prod_{\sigma\in S_1}\left[\frac{\sigma(p)}{\sigma(q)} \pm\frac{\varepsilon}{|\sigma(q)|\|q\|_{\rr}^{r_{\sigma}}}\right]\times
 \prod_{\sigma\in S_2}\left[\frac{\sigma(p)}{\sigma(q)} \pm\frac{\varepsilon}{|\sigma(q)|}\right]\subset \prod_{\sigma\in S}\RR,$$
 where $[A\pm B]$ denotes the interval $[A-B,A+B]\subset \RR$. Then set
  \begin{equation}\label{def-badc}
  \Bad_{\varepsilon}(K,\rr):=\prod_{\sigma\in S}\RR \setminus \bigcup_{(p,q)\in \cO_K\times \cO_K(\rr,\varepsilon)}\Delta_{\varepsilon}(p,q).
  \end{equation}
  It is not hard to check:
  \begin{lemma}
  $$\Bad(K,\rr)=\bigcup_{\varepsilon>0}\Bad_{\varepsilon}(K,\rr).$$
  \end{lemma}
  \begin{proof}
  It suffices to show that the set of vectors $\bx=(x_{\sigma})_{\sigma\in S} \in \prod_{\sigma\in S}\RR$ satisfying
  \begin{equation}\label{def-bad-epsilon}
\inf_{\substack{q\in \cO_K\setminus \{0\}\\p\in \cO_K }} \max \left\{\max_{\sigma\in S_1} \|q\|_{\rr}^{r_{\sigma}}|\sigma(q)x_{\sigma}+\sigma(p)|,
\max_{\sigma\in S_2} \max\{|\sigma(q)x_{\sigma}+\sigma(p)|,|\sigma(q)|\} \right\}>\varepsilon
\end{equation}
coincides with $\Bad_{\varepsilon}(K,\rr)$. By the definition of $\cO_K(\rr,\epsilon)$, equation \eqref{def-bad-epsilon} is equivalent to the following
\begin{equation}\label{def-bad-eps}
\inf_{\substack{q\in \cO_K(\rr,\varepsilon)\\p\in \cO_K }} \max \left\{\max_{\sigma\in S_1} \|q\|_{\rr}^{r_{\sigma}}|\sigma(q)x_{\sigma}+\sigma(p)|,
\max_{\sigma\in S_2}|\sigma(q)x_{\sigma}+\sigma(p)| \right\}>\varepsilon.
\end{equation}
 Now we are reduced to show the set of vectors $\bx=(x_{\sigma})_{\sigma\in S} \in \prod_{\sigma\in S}\RR$ satisfying \eqref{def-bad-eps} coincides with $\Bad_{\varepsilon}(K,\rr)$, which is straightforward to verify, and hence omitted.
  \end{proof}

To prove the set $\Bad(K,\rr)$ is HAW, it suffices to prove that it is $(\beta,\gamma)$-hyperplane potential winning for any
$\beta\in (0,1),\gamma>0.$ We choose and fix a pair of such $(\beta,\gamma)$ in this section. Furthermore, we denote the ball chosen by Bob in the first round of the game by $B_0$. By letting Alice making arbitrary moves at the first rounds and relabelling the index, we may assume $\rho_0=\rho(B_0)<1$ without loss of generality. Choose and fix $R>0$ satisfying
\begin{equation}\label{value-R}
\frac{d}{R^{\gamma}-1}\le \left(\frac{\beta^2}{2}\right)^{\gamma}.
\end{equation}
Then set
\begin{equation}\label{def-values}
\varepsilon=\frac{1}{4}\rho_0R^{-4d} \quad \text{ and } \quad H_{n}=\varepsilon\rho_0^{-1} R^n \quad  (n\ge 1).
\end{equation}
Now for $n\ge 0$, we define a class of closed balls $\sB_n$ as
$$\sB_n:=\{B\subset B_0: \beta R^{-n}\rho_0<\rho(B)\le R^{-n}\rho_0\}.$$

We are going to define a subdivision of $\cO_K(\rr,\varepsilon)$. To begin, we shall need the following height function:
$$H: \cO_K(\rr,\varepsilon)\to \RR,\quad H(q)= \max_{\sigma\in S_1} |\sigma(q)|\|q\|_{\rr}^{r_{\sigma}}.$$
We have the following lemma controlling the sizes of $H(q)$ and $\|q\|_{\rr}$.
\begin{lemma} For all $q\in \cO_K(\rr,\varepsilon)$, there holds
\begin{equation}\label{ine-hq-qr}
1\le \|q\|_{\rr}^{\frac{1}{d}}\le H(q)\le \|q\|_{\rr}^{2r}.
\end{equation}
\end{lemma}

\begin{proof}
For the second inequality in \eqref{ine-hq-qr}, we have
\begin{equation*}
H(q)^{d_1}\ge\prod_{\sigma\in S_1} |\sigma(q)|\|q\|_{\rr}^{r_{\sigma}}\ge \left(\prod_{\sigma\in S_2}\sigma(q)\right)^{-1}|N(q)|\|q\|_{\rr} \ge \|q\|_{\rr}.
\end{equation*}
The third inequality in \eqref{ine-hq-qr} is a direct consequence of the following estimate
\begin{equation}\label{ine-q-qr}
|\sigma(q)|\le \|q\|_{\rr}^{r_{\sigma}}, \quad \text{ for all } \sigma\in S_1,
\end{equation}
which is easy to check by the definition of $\|q\|_{\rr}$.
Finally, according to \eqref{ine-q-qr}, we have
\begin{equation*}
\|q\|_{\rr}\ge \prod_{\sigma\in S_1} |\sigma(q)|\ge \left(\prod_{\sigma\in S_2}\sigma(q)\right)^{-1}|N(q)|\ge 1.
\end{equation*}
This gives the first inequality.
\end{proof}

Now we can define the subdivision of $\cO_K(\rr,\varepsilon)$. Set
\begin{equation*}
\cP_n=\{q\in \cO_K(\rr,\varepsilon): H_n\le H(q)< H_{n+1}\},
\end{equation*}
and
\begin{equation*}
\cP_{n,k}=\{q\in \cP_n: H_nR^{(4k-4)d}\le \|q\|_{\rr}^{2r}< H_n R^{4kd}\}.
\end{equation*}

In view of \eqref{ine-hq-qr} and the trivial estimate $H_1<1$, we have
$$\cO_K(\rr,\varepsilon)=\bigcup_{n\ge 0}\cP_n.$$
The following lemma is important.
\begin{lemma}
$$\cO_K(\rr,\varepsilon)=\bigcup_{n\ge0}\bigcup_{k\ge 1} \cP_{n+k,k}.$$
\end{lemma}
\begin{proof}
To prove this lemma, it is  equivalent to prove that
 \begin{equation}\label{claim-nk}
 \cP_{n,k}=\emptyset \qquad \text { for all }k\ge n.
 \end{equation}
  Assuming to the contrary that there is $q\in\cP_{n,k}$ for some $k\ge n$, then we have
  \begin{equation*}
   \|q\|_{\rr}^{2}\ge \|q\|_{\rr}^{2r} \ge H_nR^{(4n-4)d} >H_{n+1}^{2d}
  \end{equation*}
  by \eqref{def-values}. This contradicts \eqref{ine-hq-qr}, hence proves \eqref{claim-nk}.
  \end{proof}

We shall need the following lemma:

\begin{lemma}\label{lemma qrank1}
Let $B\in \sB_n$. Then for any $k\ge 1$, the map $F: \cO_K\times\cO_K(\rr,\varepsilon) \to K^*$ defined by
$$F(p,q)=\frac{p}{q}$$
 is constant on the set
\begin{equation*}
\cP_{n+k,k}(B):=\{(p,q): q\in \cP_{n+k,k} \text{ and } \Delta_{\varepsilon}(p,q)\cap B \ne \emptyset\}.
\end{equation*}
\end{lemma}

\begin{proof}
For any $B\in \sB_n$ and $q\in \cP_{n+k,k}$, we have
\begin{equation}\label{e:compare}
\rho(B)\le  \frac{R^{k+1}\varepsilon}{H(q)}
\end{equation}
by \eqref{def-values}. Suppose the contrary that we have two pairs $(p_1,q_1)$ and $(p_2,q_2)$ with
\begin{equation}\label{e:noteq}
\frac{p_1}{q_1}\ne \frac{p_2}{q_2}
\end{equation}
satisfying
\begin{equation}\label{e:intersect}
\Delta_{\varepsilon}(p_1,q_1)\cap B \ne \emptyset \text{ and } \Delta_{\varepsilon}(p_2,q_2)\cap B \ne \emptyset.
\end{equation}
Then it follows from \eqref{e:noteq} that
\begin{equation}\label{e:arith}
\left|\prod_{\sigma\in S} \left(\frac{\sigma(p_1)}{\sigma(q_1)}-\frac{\sigma(p_2)}{\sigma(q_2)}\right)\right|
\ge \left|\frac{N(p_1q_2-p_2q_1)}{N(q_1q_2)}\right| \ge \frac{1}{|N(q_1q_2)|}.
\end{equation}

Now we claim that we can also prove the following inequality
\begin{equation}\label{e:dist}
\left|\prod_{\sigma\in S} \left(\frac{\sigma(p_1)}{\sigma(q_1)}-\frac{\sigma(p_2)}{\sigma(q_2)}\right)\right|<\frac{1}{|N(q_1q_2)|},
\end{equation}
which contradicts \eqref{e:arith}, hence completes the proof of the lemma.
Indeed it follows from \eqref{e:intersect} and the definition of $\Delta_{\varepsilon}(p,q)$ that, for all $\sigma\in S_1$, we have
\begin{equation}\label{e:dist-component}
\left|\frac{\sigma(p_1)}{\sigma(q_1)}-\frac{\sigma(p_2)}{\sigma(q_2)}\right|\le
\frac{\varepsilon}{|\sigma(q_1)|\|q_1\|_{\rr}^{r_{\sigma}}}+\frac{\varepsilon}{|\sigma(q_2)|\|q_2\|_{\rr}^{r_{\sigma}}}+2\rho(B).
\end{equation}
In view of \eqref{e:compare} and \eqref{e:dist-component}, we have
\begin{align}\label{ine-s1}
 \left|\prod_{\sigma\in S_1} \left(\frac{\sigma(p_1)}{\sigma(q_1)}-\frac{\sigma(p_2)}{\sigma(q_2)}\right)\right|
 &\le \prod_{\sigma\in S_1} \left(\frac{\varepsilon}{|\sigma(q_1)|\|q_1\|_{\rr}^{r_{\sigma}}}+\frac{\varepsilon}{|\sigma(q_2)|\|q_2\|_{\rr}^{r_{\sigma}}}+2\rho(B)\right)
 \notag \\
 &\le \prod_{\sigma\in S_1} \left(\frac{\varepsilon}{|\sigma(q_1)|\|q_1\|_{\rr}^{r_{\sigma}}}+\frac{\varepsilon}{|\sigma(q_2)|\|q_2\|_{\rr}^{r_{\sigma}}}+
  \frac{2R^{k+1}\varepsilon}{\max\{H(q_1), H(q_2)\} }\right)  \notag\\
 &\le  (R^{k+1}+1)^d \prod_{\sigma\in S_1} \frac{\varepsilon}{|\sigma(q_1q_2)|}
 \left(\frac{|\sigma(q_1)|}{\|q_2\|_{\rr}^{r_{\sigma}}}+\frac{|\sigma(q_2)|}{\|q_1\|_{\rr}^{r_{\sigma}}}\right) \notag \\
 &\le 2^dR^{dk+d}  \frac{\varepsilon^d}{|\prod_{\sigma\in S_1}\sigma(q_1q_2)|}
 \prod_{\sigma\in S_1}R^{4r_{\sigma}d}\left(\frac{|\sigma(q_1)|}{\|q_1\|_{\rr}^{r_{\sigma}}}+\frac{|\sigma(q_2)|}{\|q_2\|_{\rr}^{r_{\sigma}}}\right)\notag\\
 &\le 2^{2d-1} R^{dk+5d}\varepsilon^d \frac{1}{|\prod_{\sigma\in S_1}\sigma(q_1q_2)|}
  \left(\frac{|\omega(q_1)|}{\|q_1\|_{\rr}^{r}}+\frac{|\omega(q_2)|}{\|q_2\|_{\rr}^{r}}\right) \notag \\
  &\le 2^{2d-1} R^{dk+5d}\varepsilon^d \frac{1}{|\prod_{\sigma\in S_1}\sigma(q_1q_2)|}
  \left(\frac{|H(q_1)|}{\|q_1\|_{\rr}^{2r}}+\frac{|H(q_2)|}{\|q_2\|_{\rr}^{2r}}\right) \notag \\
 &\le 2^{2d} R^{dk+5d}R^{-(4k-4)d}\varepsilon^d\frac{1}{|\prod_{\sigma\in S_1}\sigma(q_1q_2)|} \notag \\
 &< \frac{1}{|\prod_{\sigma\in S_1}\sigma(q_1q_2)|}.
 \end{align}
 On the other hand,  it follows  from \eqref{e:intersect} and the definition of $\Delta_{\varepsilon}(p,q)$ that, for all $\sigma\in S_2$, we have
\begin{equation}\label{e:dist-component2}
\left|\frac{\sigma(p_1)}{\sigma(q_1)}-\frac{\sigma(p_2)}{\sigma(q_2)}\right|\le
\frac{\varepsilon}{|\sigma(q_1)}+\frac{\varepsilon}{|\sigma(q_2)|}+2\rho(B).
\end{equation}
In view of \eqref{e:dist-component2} and the assumption $\rho_0<1$, we have
\begin{align}\label{ine-s2}
\left|\prod_{\sigma\in S_2} \left(\frac{\sigma(p_1)}{\sigma(q_1)}-\frac{\sigma(p_2)}{\sigma(q_2)}\right)\right|
&\le \prod_{\sigma\in S_2}\left(\frac{\varepsilon}{|\sigma(q_1)|}+\frac{\varepsilon}{|\sigma(q_2)|}+2\rho(B)\right) \notag\\
&\le \prod_{\sigma\in S_2} \frac{4\varepsilon^2}{|\sigma(q_1q_2)|} <\frac{1}{\prod_{\sigma\in S_2}|\sigma(q_1q_2)|}.
\end{align}
Note that we have used the fact that $|\sigma(q)|\le \varepsilon$ for $\sigma\in S_2$ and $q\in \cO_K(\rr,\varepsilon)$ and an elementary inequality saying that $4ab\ge a+b+2$ for $a,b\ge 1$.
Now \eqref{e:dist} follows from \eqref{ine-s1} and \eqref{ine-s2}. Hence our proof is completed.
\end{proof}

Now we are in a position to prove Proposition \ref{main lemma}.

\begin{proof}[Proof of Proposition  \ref{main lemma}.]
For any $B\in \sB_n$ and $k\ge 1$, denote the unique point given by Lemma \ref{lemma qrank1} as
$$\bs(k,B)=(s_{\sigma}(k,B))_{\sigma\in S}.$$
Then it follows from Lemma \ref{lemma qrank1} and the definition of $\cP_{n+k,k}$ that
\begin{equation*}
\bigcup_{(p,q)\in \cP_{n+k,k}}\Delta_{\varepsilon}(p,q) \cap B \subset \bigcup_{\tau\in S}E_{\tau}(k,B)^{(R^{-n-k}\rho_0)},
\end{equation*}
where the hyperplane $E_{\tau}(k,B)$ is defined as
\begin{equation}\label{def-hyperplane}
E_{\tau}(k,B):=\{\bx\in \prod_{\sigma\in S}\RR: x_{\tau}=s_{\tau}(k,B) \}.
\end{equation}

As those $\sB_n$ are mutually disjoint, for each $i\ge 0$ there exists at most one $n\geq 0$ with $B_{i}\in \sB_n$. According to the definition of $(\beta,\gamma)$-hyperplane potential game, we have $\rho_{i+1}\geq\beta \rho_{i}$. In view of \cite[Remark 2.4]{AGK}, we may assume that $\rho_0\rightarrow 0$.
Hence for each $n\geq 0$, there exists an $i\geq 0$ with $B_{i}\in\sB_n$. Let
$i(n)$ denote the smallest $i$ with $B_{i}\in\sB_n$. Then, the map
$n\mapsto i(n)$ is an injective one from $\ZZ_{\geq 0}$ to
$\mathbb{Z}_{\geq 0}$. Let Alice play according to the following strategy: each time after
Bob chooses a closed ball $B_i$, if $i=i(n)$ for some $n\geq 0$, then Alice chooses
the family of hyperplane neighborhoods
$$\{E_{\tau}(k,B_{i(n)})^{(R^{-n-k}\rho_0)}:\tau\in S, k\in\NN\}.$$
where $E_{\tau}(k,B_{i(n)})$ is the hyperplane given by \eqref{def-hyperplane}. Otherwise Alice makes
an arbitrary move. Since $B_{i(n)}\in\sB_n$, $\rho_{i(n)}>\beta R^{-n}\rho_0$. Then,
 \eqref{value-R} implies that
$$\sum_{\tau\in S, k=1}^\infty(R^{-n-k}\rho_0)^\gamma=d(R^{-n}\rho_0)^\gamma(R^\gamma-1)^{-1}
\leq \left(\frac{\rho_{i}}{\beta}\right)^{\gamma}\left(\frac{\beta^2}{2}\right)^{\gamma} < (\beta \rho_{i})^\gamma.$$
Hence  Alice's move
is legal.
Then we have
\begin{align*}
\bigcap_{i=0}^\infty B_i
& =\bigcap_{i=0}^{\infty}B_i\cap\Big(\Bad(K,\rr)\cup \bigcup_{(p,q)\in \cO_K\times\cO_K(\rr,\epsilon)}\Delta_{\varepsilon}(p,q) \Big) \\
& =\bigcap_{i=0}^{\infty}B_i\cap\Big(\Bad(K,\rr)\cup \bigcup_{n=0}^{\infty}
\bigcup_{k=1}^{\infty}\bigcup_{q\in \cP_{n+k,k}}\bigcup_{p\in \cO_K}\Delta_{\varepsilon}(p,q)\Big)\\
& \subset \Bad(K,\rr)\cup\Big(\bigcup_{n=0}^{\infty}\bigcup_{k=1}^{\infty}
\bigcup_{q\in \cP_{n+k,k}}\bigcup_{p\in \cO_K}\Delta_{\varepsilon}(p,q)\cap B_{i(n)}\Big)\\
& =\Bad(K,\rr)\cup\Big(\bigcup_{n=0}^{\infty}\bigcup_{k=1}^{\infty}
\bigcup_{(p.q)\in \cP_{n+k,k}(B_{i(n)})}\Delta_{\varepsilon}(p,q)\cap B_{i(n)}\Big)\\
& \subset \Bad(K,\rr)\cup \Big(\bigcup_{n=0}^{\infty}\bigcup_{k=1}^\infty \bigcup_{\tau\in S} E_{\tau}(k,B_{i(n)})^{(R^{-n-k}\rho_0)}\Big).
\end{align*}
Thus the unique point $\bx_\infty\in\bigcap_{i=0}^\infty B_i$ lies
in $$\Bad(K,\rr)\cup \Big(\bigcup_{n=0}^{\infty}\bigcup_{k=1}^\infty \bigcup_{\tau\in S}
E_{\tau}(k,B_{i(n)})^{(R^{-n-k}\rho_0)}\Big).$$
Hence, Alice wins.
\end{proof}

\section{Proof of the main theorem}
We shall need the following simple observation.
\begin{lemma}\label{dec-irr}
Let $\Gamma$ and $\Gamma'$ be lattices in $G$ such that $\Gamma$ is commensurable with $\Gamma'$. Then for any subgroup $F$ of $G$, there holds
\begin{equation*}
\text{$E(F)$ is HAW on $\ggm \Longleftrightarrow$ $E(F)$ is HAW on $ G/\Gamma'$}.
\end{equation*}
\end{lemma}


\begin{proof}

As $\Gamma, \Gamma'$ are commensurable with each other, the group $\Gamma''=\Gamma\cap \Gamma'$  is of finite index in both $\Gamma$ and $\Gamma'$,  and hence is a lattice subgroup of $G$. By replacing $\Gamma'$ with $\Gamma''$, the proof of the lemma can be reduced to the case when $\Gamma'\subset \Gamma$. In this case, the natural projection map $\pi:G/\Gamma\mapsto G/\Gamma'$ is a finite covering map. Now the lemma follows from Lemma \ref{haw-manifold}.
\end{proof}

\begin{proof}[Proof of Theorem \ref{main thm}]
Let $G$ be a product of copies of $\SL_2(\RR)$ and $\Gamma$ a lattice. Then according to \cite[Theorem 5.22]{Rag}, there exist  $G_i \ (1\le i\le k)$ which are products of copies of $\SL_2(\RR)$ such that $\prod_{i=1}^{k} G_i=G$ and irreducible lattices $\Gamma_i\subset G_i$ such that
 the lattice $\Gamma$ is commensurable with $\Gamma_1\times\cdots\times \Gamma_k$.  In view of Lemma \ref{dec-irr}, we are reduced to consider the case when  $\Gamma=\Gamma_1\times\cdots\times \Gamma_k$. Moreover, since an orbit is bounded on $\ggm$ if and only if its projection is bounded on each $G_i/\Gamma_i$, we are reduced to consider the case when $\Gamma$ itself is irreducible and not cocompact by applying Lemma \ref{haw-manifold} (5).
Now there are two cases:

\textbf{Case 1.} Suppose $G=\SL_2(\RR)$. This is done explicitly in \cite[Theorem 3.7]{KW3}.

\textbf{Case 2.} Suppose $G$ is a  product of more than two copies of $\SL_2(\RR)$. Then it follows from Margulis arithmeticity theorem \cite[Chapter IX, Theorem 1.9A]{Mar} that this $\Gamma$ is arithmetic, i.e., $\Gamma$ is commensurable with $\bG(\ZZ)$ with $\bG$ a $\QQ$-simple semisimple group. Then $\bG=\Res_{K/\QQ}\bG'$ with $\bG'$ a $K$-form of $\SL_2$ for some totally real field $K$. Since $\Gamma$ is not cocompact, we have $\bG'$ is $K$-isotropic. Hence $\bG'=\mathbf{\SL_2}$ and  $\Gamma$ is commensurable with $\Res_{K/\QQ}\SL_2(\ZZ)$. Now let $F=\{g_t: t\in \RR\}$ be any one-parameter $\Ad$-semisimple subgroup of $G$. By the real Jordan decomposition, we know that $g_t=g_t'g_t''$ where $g_t'$ is diagonalizable over $\RR$ and $\{g_t'' : t\in \RR\}$ is bounded. Then it is clear that $E(F)=E(F')$ where $F'=\{g_t': t\in \RR\}$. It is easily checked that there exist $h\in G$ and $\rr$ such that $F'=hF_\rr h^{-1}$. Thus, we have $E(F')=hE(F_\rr)$. Then it follows from Lemma  \ref{haw-manifold}(3) and Proposition \ref{main prop}, that $E(F)$ is HAW. This completes the proof.
\end{proof}

\end{document}